\documentclass{article}%
\usepackage{amsmath}%
\usepackage{amsfonts}%
\usepackage{amssymb}%
\usepackage{graphicx}
\providecommand{\U}[1]{\protect\rule{.1in}{.1in}}
\newtheorem{theorem}{Theorem}

\newtheorem{corollary}[theorem]{Corollary}

\newtheorem{definition}[theorem]{Definition}

\newtheorem{lemma}[theorem]{Lemma}

\newtheorem{proposition}[theorem]{Proposition}
\newtheorem{remark}[theorem]{Remark}

\newenvironment{proof}[1][Proof]{\noindent\textbf{#1.} }{\ \rule{0.5em}{0.5em}}
\begin{document}

\title{2D Smagorinsky type large eddy models as limits of stochastic PDEs}
\author{Franco Flandoli\footnote{Email: franco.flandoli@sns.it. Scuola Normale Superiore of Pisa, Piazza dei Cavalieri 7, 56124 Pisa, Italy}\quad
Dejun Luo\footnote{Email: luodj@amss.ac.cn. Key Laboratory of RCSDS, Academy of Mathematics and Systems Science, Chinese Academy of Sciences, Beijing 100190, China and School of Mathematical Sciences, University of Chinese Academy of Sciences, Beijing 100049, China} \quad
Eliseo Luongo\footnote{Email: eliseo.luongo@sns.it. Scuola Normale Superiore of Pisa, Piazza dei Cavalieri 7, 56124 Pisa, Italy}}
\maketitle

\begin{abstract}
We prove that a version of Smagorinsky Large Eddy model for a 2D\ fluid in
vorticity form is the scaling limit of suitable stochastic models for large scales,
where the influence of small turbulent eddies is modeled by a transport type noise.
\end{abstract}

\textbf{Keywords:} Smagoringsky model, eddy viscosity model, turbulence, transport noise, scaling limit

\textbf{MSC (2020):} 60H15, 76D05

\section{Introduction}

Recently, a new stochastic approach has been developed in \cite{flandoli2021scaling, flandoli2022eddy, FlaLuoWaseda, FlandoliOberwolfach, Fla23, debussche2022second, carigi2023dissipation} to
explain Boussinesq hypothesis that ``turbulent fluctuations are dissipative on
large scales'' \cite{boussinesq1877essai}. The idea, better explained below in
Section \ref{section heuristic}, is that the large scales satisfy a
Navier-Stokes type equation with a stochastic transport term corresponding to
the action of small scales. In a suitable scaling limit, we get a
deterministic Navier-Stokes equation with an additional dissipative term. The
turbulent viscosity is directly related to the noise (namely small-scale)
covariance. All the quoted works are related to dimension 2, with the
exception of \cite{Fla23} that deals with a 2D-3C model with some three
dimensional feature, including a stretching term of small scales over large
ones and the possibility of an AKA (anisotropic kinetic alpha) effect in the limit equation. For other
approaches to justify Boussinesq hypothesis and turbulent viscosity based on Eulerian formulations of fluid dynamical systems see for
instance \cite{BerselliLES, jiang2020foundations, wirth1995eddy}. There are also different models based on filtering the systems at the Lagrangian level rather then Eulerian one, we refer to \cite{foias2001navier, foias2002three, cheskidov2005leray, cotter2017stochastic} for rigorous analysis and some discussions on the topic.

The previous works on the stochastic approach are, however, limited to the
case of \textit{linear} limit dissipation term, namely turbulent viscosity
independent of the solution. Smagorinsky type models are excluded from the
previous analysis and it was not clear for some time how to incorporate them
into this new theory. In this paper we solve this problem. This provides new
insight into these models and their motivations.

Since our techniques are, at present, well developed for the vorticity
equation, while they suffer certain difficulties for the velocity equation,
we present the results for vorticity type equations (however, as stated in \cite[Section 5]{Win96}, the performances of vorticity-velocity models are sometimes superior to those of velocity-pressure ones). We choose the following
form, discussed for instance in \cite{cottet2003vorticity}:%
\begin{equation}
\partial_{t}\omega_{L}+u_{L}\cdot\nabla\omega_{L}=\nu\Delta\omega
_{L}+\operatorname{div}\left(  g^{\prime}\left(  \omega_{L}\right)
\nabla\omega_{L}\right)  \label{smagorinsky introduction}%
\end{equation}
(written in this way so that $\operatorname{div}\left(  g^{\prime}\left(
\omega_{L}\right)  \nabla\omega_{L}\right)  =\Delta g\left(  \omega
_{L}\right)  $) with the additional conditions $\omega_{L}=\nabla^{\perp} \cdot
u_{L}$, $\operatorname{div}u_{L}=0$ and the initial condition $\omega
_{L}|_{t=0}=\omega_{0}^{L}$. Here, $L$ stands for the large scale components of fluid vorticity and velocity,
see the next section for more discussions; the fields are assumed to be periodic, on a torus.
The function $g\left(  r\right)  $ is subject to quite general assumptions
which include it is non-decreasing, so that $g^{\prime}$ is not negative. The particular case treated in \cite{cottet2003vorticity} (see also \cite{Man78, Win96, Deng20}) is
\begin{align} \label{eq:g-prime}
    g'(r)=(C_s\mathbf{\Delta})^2 \lvert r\rvert,
\end{align}
where $\mathbf{\Delta}$ is a subgrid characteristic length-scale and $C_s$ is a non-dimensional constant which has to be calibrated and its value may vary with the type of the flow and the Reynolds number. However, similarly to the Smagorinsky model in velocity form, it may be useful to cover more general nonlinearities, see for instance \cite[Section 3.3.2]{BerselliLES}. We
prove that this Smagorinsky type model is the limit of the large-scale
stochastic model%

\begin{equation}
d\omega_{L}+\left(  u_{L}\cdot\nabla\omega_{L}-\nu\Delta\omega_{L}\right)
dt=-f^{\prime}\left(  \omega_{L}\right)  \sum_{k}\sigma_{k}\cdot
\nabla\omega_{L}\circ dW_{t}^{k} \label{stochastic approximation introduction}%
\end{equation}
(where again $f^{\prime}\left(  \omega_{L}\right)  \sigma_{k}\cdot
\nabla\omega_{L}=\sigma_{k}\cdot\nabla f \left(  \omega_{L}\right)  $)
with $f$ such that $\frac{1}{4}f^{\prime}\left(  r\right)  ^{2}=g^{\prime}\left(
r\right)  $. The limit is taken along a suitable sequence of small-scale
noise, namely we assume (roughly speaking) that $\sigma_{k}$ are smaller and
smaller scale (an assumption of scale separation). The notations and
assumptions (like the fact that $\{W^{k}\}_k$ are independent Brownian motions
and $\circ$ is the Stratonovich multiplication operation) will be explained in
the technical sections.

The paper is organized as follows. In Section \ref{section heuristic} we
describe the heuristic ideas behind the stochastic model. In Section
\ref{Sec preliminaries} we state our results and introduce some mathematical
tools. In Section \ref{sec well posed} we show the existence of martingale
solutions of the problem (\ref{stochastic approximation introduction}) above.
Lastly, in Section \ref{sect scaling limit} we will show our main result about
the convergence of martingale solutions of our stochastic models to a measure concentrated on
the unique weak solution of the Smagorinsky model
(\ref{smagorinsky introduction}), see Theorem \ref{Scaling Limit Smagorinski}
below for the rigorous statement.

\section{The heuristic idea\label{section heuristic}}

The idea described in this section is similar to the one given in
\cite{FlandoliOberwolfach, Fla23}, but we repeat it and particularise
the models studied here, for completeness and to help the intuition behind the
model. Consider a 2D Newtonian viscous fluid in a torus, described in
vorticity form by the equations%
\begin{align*}
\partial_{t}\omega+u\cdot\nabla\omega &  =\nu\Delta\omega, \\
\omega =\nabla^{\perp}\cdot u,\quad {\rm div}\, u &=0,  \\
\omega|_{t=0}  &  =\omega_{0} ,
\end{align*}
where $\omega$ is the vorticity
field and $u$ the velocity field. Assume that the initial vorticity
$\omega_{0}$ is the sum of a large scale component $\omega_{0}^{L}$ plus a
small-scale component $\omega_{0}^{S}$. Then, at least on a short time
interval $[0,\tau ]$, it is reasonable to expect that the system%
\begin{align*}
\partial_{t}\omega_{L}+u\cdot\nabla\omega_{L}  &  =\nu\Delta\omega_{L}, \\
\partial_{t}\omega_{S}+u\cdot\nabla\omega_{S}  &  =\nu\Delta\omega_{S} ,\\
\omega_{L}|_{t=0} & =\omega_{0}^{L},\quad\omega_{S}|_{t=0} =\omega_{0}^{S}%
\end{align*}
represents quite well the evolution of the different vortex structures, as for
instance in the small vortex-blob limit to point vortices treated by
\cite{marchioro2012mathematical}. The system above is equivalent to the
original one, by addition.

The next step is considering only the equation for the large scales, isolating
the term which is not closed, namely depends on the small scales:%
\[
\partial_{t}\omega_{L}+u_{L}\cdot\nabla\omega_{L}-\nu\Delta\omega_{L}%
=-u_{S}\cdot\nabla\omega_{L}.
\]
Here $u_{L}$, with $\operatorname{div}u_{L}=0$, has the property
$\nabla^\perp \cdot u_{L}=\omega_{L}$ (namely $u_{L}$ is reconstructed from
$\omega_{L}$ by Biot-Savart law). The field $u_{S}$ should correspond to
$\omega_{S}$ by Biot-Savart law but we now introduce a stochastic closure
assumption. We replace $u_{S}\left(  t,x\right)  $ by a white-in-time noise,
with suitable space dependence%

\[
u_{S}\left(  t,x\right)  \mapsto \chi\left(  t,x\right)  \sum_{k}\sigma
_{k}\left(  x\right)  \frac{dW_{t}^{k}}{dt}, %
\]
where $\{\sigma_{k} \}_k$ are suitable divergence free vector fields, and $\chi(t,x)$ is a scalar stochastic process which will be linked to the large scales, in order to model the idea that the turbulent small scales are more active where the large scales have more intense variations (e.g. larger shear); $\{W^{k} \}_k$ are
independent scalar Brownian motions. In the replacement, Stratonovich
integrals are used, in accordance with the Wong-Zakai principle (see rigorous
results in \cite{debussche2022second}). Therefore the equation for large
scales, now closed and stochastic, takes the form%
\[
d\omega_{L}+\left(  u_{L}\cdot\nabla\omega_{L}-\nu\Delta\omega_{L}\right)
dt=- \chi\left(  t,x\right)  \sum_{k}\sigma_{k}\cdot\nabla\omega_{L}\circ
dW_{t}^{k}.
\]
Previous works developed this idea in the case when $\chi=1$, see e.g. \cite{FlaLuoWaseda, flapapp2022fromadditive, debussche2022second}. Here we assume
that $\chi$ is a function of $\omega_{L}$, that for notational convenience
will be written as%
\[
\chi\left(  t,x\right)  =f^{\prime}\left(  \omega_{L}\left(
t,x\right)  \right)
\]
for a suitable function $f$. As said above, the heuristic idea is that turbulence is more
developed in regions of high large-scale vorticity, hence the small-scale
noise should be modulated by an increasing function $f^{\prime}$.

This is the motivation for the stochastic model
(\ref{stochastic approximation introduction}) presented in the Introduction.
Our main purpose is showing that it leads to the Smagorinsky type
deterministic equation (\ref{smagorinsky introduction}) in a suitable scaling
limit of the noise.

\section{Functional Setting and Main Results}\label{Sec preliminaries}
Let us set some notation before stating the main contributions of this work. Let $\mathbb{T}^2=\mathbb{R}^2/\mathbb{Z}^2$ be the two dimensional torus and $\mathbb{Z}_0^2=\mathbb{Z}^2\setminus \{0\}$ the nonzero lattice points. Let $(H^{s,p}(\mathbb{T}^2), \lVert\cdot\rVert_{H^{s,p}}),\ s\in\mathbb{R},\ p\in (1,+\infty)$ be the Bessel spaces of zero mean periodic functions. In case of $p=2$, we simply write $H^{s}(\mathbb{T}^2)$ in place of $H^{s,2}(\mathbb{T}^2)$ and we denote by $\langle \cdot,\cdot\rangle_{H^s}$ the corresponding scalar products. In case also $s>0$ we denote by $\langle \cdot,\cdot\rangle_{H^{-s},H^s}$ the dual pairing between $H^s$ and $H^{-s}$. Lastly we denote by $H^{s-}(\mathbb{T}^2)=\cap_{r<s}H^r(\mathbb{T}^2)$.
In case of $s=0$ we will write $L^2(\mathbb{T}^2)$ instead of $H^0(\mathbb{T}^2)$ and we will neglect the subscript in the notation for the norm and the inner product.
Similarly, we introduce the Bessel spaces of zero mean vector fields
\begin{align*}
    \mathbf{H}^{s,p}&= \{(u_1,u_2)^t:\ u_1,u_2\in H^{s,p}(\mathbb{T}^2) \},\\  \langle u, v\rangle_{\mathbf{H}^s}&=\langle u_1,v_1\rangle_{H^s}+\langle u_2,v_2\rangle_{H^s}, \quad \text{for } s\in \mathbb{R}.
\end{align*}
Again, in case of $s=0$ we will write $\mathbf{L}^2$ instead of $\mathbf{H}^0$ and we will neglect the subscript in the notation for the norm and the scalar product.

Let $Z$ be a separable Hilbert space, with associated norm $\| \cdot\|_{Z}$. We denote by $C^w_{\mathcal{F}}\left(  \left[  0,T\right]  ;Z\right) $ the space of weakly continuous adapted processes $\left(  X_{t}\right)  _{t\in\left[
0,T\right]  }$ with values in $Z$ such that
\[
\mathbb{E} \bigg[ \sup_{t\in\left[  0,T\right]  }\left\Vert X_{t}\right\Vert
_{Z}^{2}\bigg]  <\infty
\]
and by $L_{\mathcal{F}}^{p}\left(  0,T;Z\right),\ p\in [1,\infty),$ the space of progressively
measurable processes $\left(  X_{t}\right)  _{t\in\left[  0,T\right]  }$ with
values in $Z$ such that
\[
    \mathbb{E} \bigg[ \int_{0}^{T}\left\Vert X_{t}\right\Vert _{Z}^{p}dt \bigg]
<\infty.
\]

Following the ideas introduced in Section \ref{section heuristic}, we are interested in the following stochastic model with a more precise noise (cf. \cite{galeati2020convergence, flandoli2021scaling})
\begin{equation}\label{introductory equation}
    \begin{cases}
        d\omega_L=(\nu \Delta \omega_L- u_L\cdot\nabla \omega_L)\, dt - \sum_{k\in \mathbb{Z}^2_0}\theta_k \sigma_k\cdot \nabla f(\omega_L)  \circ dW^k, \\
        u_L=-\nabla^{\perp}(-\Delta)^{-1}\omega_L, \\
        \omega_L(0)=\omega_0,
    \end{cases}
\end{equation}
where $f\in C^{1}(\mathbb R;\mathbb R)$, $\theta=(\theta_k)_{k} \in \ell^2(\mathbb Z^2_0)$ satisfies
  \begin{equation}\label{eq:theta}
  \sum_{k\in \mathbb{Z}^2_0} \theta_k^2=1, \quad  \theta_k=\theta_l\quad \mbox{if } |k|=|l|;
  \end{equation}
$\{\sigma_k \}_{k\in \mathbb{Z}^2_0}$ is the standard orthonormal basis of divergence free vector fields in $\mathbf{L}^2$ made by the eigenfunctions of the Stokes operator, i.e.
\begin{align*}
    \sigma_k(x)= \frac{k^{\perp}}{\lvert k\rvert}e_k(x),\quad e_k(x) = \sqrt{2}\begin{cases}
        \cos(2\pi k\cdot x)\quad \text{if } k\in \mathbb{Z}^2_+, \\
        \sin(2\pi k\cdot x)\quad \text{if } k\in \mathbb{Z}^2_-,
    \end{cases}
\end{align*}
where $k^\perp=(k_2, -k_1)$, $\mathbb{Z}^2_+:=\{k\in \mathbb{Z}_0^2: (k_1>0)\text{ or }(k_1=0,\ k_2>0)\}$ and $\mathbb{Z}^2_-:=-\mathbb{Z}^2_+$; $\{W^k \}_{k\in \mathbb{Z}^2_0}$ is a family of real independent Brownian motions. Moreover we assume that
\begin{align}\label{bound f''}
    \lvert f'(x)\rvert \lesssim 1+\lvert x\rvert^{\alpha}
\end{align}
for some $\alpha\in [0,1]$. This implies in particular that
\begin{align}
\lvert f(x)\rvert &\lesssim 1+\lvert x\rvert^{\alpha+1}\label{bound f'} .
\end{align}

In the sequel, we shall omit the subscript $L$ to save notation. System \eqref{introductory equation} can be formulated easily in It\^o form. Indeed, it holds
\begin{align}\label{Ito Stratonovich corrector}
    \sigma_k\cdot \nabla f(\omega)  \circ dW^k&= \notag \sigma_k\cdot \nabla f(\omega) \, dW^k+\frac{1}{2}d\left[\sigma_k\cdot \nabla f(\omega) , W^k\right];
\end{align}
since
  $$\aligned
  d\big(\sigma_k\cdot \nabla f(\omega)\big) &= \sigma_k\cdot \nabla\big(d f(\omega)\big) \\
  &=\sigma_k\cdot \nabla ( f'(\omega)V ) \, dt  - \sum_{l } \theta_l \sigma_k\cdot \nabla\big( f'(\omega)^2 \sigma_l\cdot \nabla \omega \big)  \circ dW^l ,
  \endaligned $$
where $V:= \nu \Delta \omega - u \cdot\nabla \omega$, one has
  $$ \aligned
  d\left[\sigma_k\cdot \nabla f(\omega) , W^k\right] &= - \theta_k \sigma_k\cdot \nabla\big( f'(\omega)^2 \sigma_k\cdot \nabla \omega \big) \, dt\\
  &= - \theta_k\, {\rm div} \big( f'(\omega)^2 (\sigma_k\otimes \sigma_k) \nabla \omega \big) \, dt
  \endaligned $$
which is due to the divergence free property of $\sigma_k$; hence,
  $$\aligned
  & - \sum_{k }\theta_k \sigma_k\cdot \nabla f(\omega) \circ dW^k\\
  &\quad = - \sum_{k}\theta_k \sigma_k\cdot \nabla f(\omega)\, dW^k + \frac{1}2 \sum_{k}\theta_k^2 \, {\rm div} \big( f'(\omega)^2 (\sigma_k\otimes \sigma_k) \nabla \omega \big) \, dt \\
  &\quad= - \sum_{k}\theta_k \sigma_k\cdot \nabla f(\omega)\, dW^k + \frac{1}4 {\rm div} \big( f'(\omega)^2 \nabla \omega \big) \, dt,
  \endaligned $$
where the last step is due to the fact (cf. \cite[Lemma 2.6]{flaluo2020enstrophy} for a proof)
  \begin{equation}\label{eq:covariance-diagonal}
  \sum_{k}\theta_k^2 \, (\sigma_k\otimes \sigma_k)= \frac12 I_2,
  \end{equation}
the latter being the $2\times 2$ unit matrix. Thanks to the computations on the It\^o-Stratonovich corrector above, equation \eqref{introductory equation} can be rewritten as
\begin{equation}\label{Ito equation prefinal}
    \begin{cases}
        d\omega = \big(\nu \Delta \omega-u\cdot\nabla \omega+\frac{1}{4}\operatorname{div}( f'(\omega)^2 \nabla \omega) \big)\, dt - \sum_{k}\theta_k \sigma_k\cdot \nabla f(\omega)\, dW^k, \\
        u =-\nabla^{\perp}(-\Delta)^{-1}\omega ,\\
        \omega(0) =\omega_0.
    \end{cases}
\end{equation}

We introduce the real function $g:\mathbb R\rightarrow \mathbb{R}$ defined as
  $$g(x)=\frac{1}{4}\int_0^x f'(t)^2\, dt,\quad x\in \mathbb R, $$
which satisfies $g(0)=0$ and
\begin{equation}\label{bound g}
\aligned
    \lvert g(y)-g(x)\rvert
    &\lesssim \lvert y-x\rvert+\left\lvert y\lvert y\rvert^{2\alpha}-x\lvert x\rvert^{2\alpha}\right\rvert  \\
    & \lesssim \lvert y-x\rvert\left(1+\lvert x\rvert^{2\alpha}\right)+\lvert y\rvert \left\lvert \lvert y\rvert^{2\alpha}-\lvert x\rvert^{2\alpha} \right\rvert.
\endaligned
\end{equation}
From the definition of $g$ it follows that system \eqref{Ito equation prefinal} can be rewritten as
\begin{equation}\label{Ito equation}
    \begin{cases}
        d\omega = \big(\nu \Delta \omega-u\cdot\nabla \omega +\Delta g(\omega) \big)\, dt - \sum_{k} \theta_k \sigma_k\cdot \nabla f'(\omega)\, dW^k,\\
        u =-\nabla^{\perp}(-\Delta)^{-1}\omega, \\
        \omega(0) =\omega_0.
    \end{cases}
\end{equation}
 The relation between $u$ and $\omega$ can be described in terms of the so-called Biot-Savart operator \begin{align*}
       K\in\mathcal{L}(H^{s,p},\mathbf{H}^{s+1,p}): \quad K[\omega]=-\nabla^{\perp}(-\Delta)^{-1}\omega \quad \text{for } p\in (1,+\infty),\ s\in\mathbb{R}.
   \end{align*}

We are now ready to define our notion of solution for system \eqref{Ito equation}.

\begin{definition}\label{well posed def}
We say that system \eqref{Ito equation} has a weak solution if there exists a filtered probability space $(\Omega,\mathcal{F},\mathcal{F}_t,\mathbb{P})$, a sequence of independent $\mathcal{F}_t$ Brownian motions $W^k$ and $\omega\in C^w_{\mathcal{F}}(0,T;L^2(\mathbb{T}^2))\cap L^2_{\mathcal{F}}(0,T;H^1(\mathbb{T}^2))$ such that for any $\phi\in C^{\infty}(\mathbb{T}^2)$, $\mathbb{P}$-a.s. $\forall\, t\in [0,T]$, it holds
\begin{align*}
    \langle \omega_t,\phi\rangle
    &= \langle \omega_0,\phi\rangle+ \nu \int_0^t \langle \omega_s,\Delta \phi\rangle \, ds+\int_0^t \langle g(\omega_s), \Delta\phi\rangle\, ds\\
    &\quad +\int_0^t \langle \omega_s, K[\omega_s]\cdot\nabla\phi\rangle\, ds + \sum_{k\in \mathbb{Z}^2_0}\int_0^t \theta_k \langle f(\omega_s), \sigma_k\cdot\nabla \phi \rangle\, dW^k_s.
\end{align*}
\end{definition}

Due to the nonlinearities appearing in equation \eqref{Ito equation} the existence of weak solutions is a nontrivial fact which will be proved in Section \ref{sec well posed}. Indeed we will prove the following result.

\begin{theorem}\label{thm well posed}
    For each $\omega_0\in L^2(\mathbb{T}^2)$ there exists at least one weak solution of system \eqref{Ito equation} in the sense of Definition \ref{well posed def}. Moreover \begin{align*}
        \sup_{t\in [0,T]}\lVert \omega_t\rVert^2 +2\nu\int_0^T \lVert \nabla \omega_s\rVert^2\, ds\leq 2\lVert \omega_0\rVert^2 \quad \mathbb{P}-a.s.
    \end{align*}
\end{theorem}

Next, following the idea introduced for the first time in \cite{galeati2020convergence}, we consider a family $\{\theta^N\}_{N\in \mathbb{N}}\subseteq \ell^2(\mathbb{Z}^2_0)$, satisfying relation \eqref{eq:theta} such that
  \begin{equation}\label{eq:theta-N}
   \lim_{N\rightarrow +\infty} \lVert \theta^N\rVert_{\ell^\infty}=0,
  \end{equation}
and we call $\omega^N$ the corresponding weak solution of equation \eqref{Ito equation} with $\{\theta^{N}_k \}_k$ in place of $\{\theta_k \}_k$. In order to complete our plan, we want to show that the law of $\omega^N$ converges weakly to a measure supported on the unique weak solution of the Navier-Stokes equation in vorticity form with Smagorinky correction, namely
\begin{equation}\label{smagorinsky equation}
    \begin{cases}
        \partial_t\overline{\omega}=\nu \Delta \overline{\omega}+\operatorname{\Delta} g(\overline{\omega})-\overline{u}\cdot\nabla \overline{\omega} ,\\
        \overline{u}=-\nabla^{\perp}(-\Delta)^{-1}\overline{\omega}, \\
        \overline{\omega}(0)=\omega_0.
    \end{cases}
\end{equation}

\begin{remark}
   Taking $f(r)=\frac{4}{3} C_s\mathbf{\Delta} \lvert r\rvert^{1/2} r$, $C_s$ and $\mathbf{\Delta}$ being the same as in \eqref{eq:g-prime}, we have $g(r)= \frac{1}{2}(C_s\mathbf{\Delta})^2 r^2 {\rm sign}(r)$, and thus $\operatorname{\Delta} g(\overline{\omega}) = (C_s\mathbf{\Delta})^2{\rm div}(|\overline{\omega}| \nabla \overline{\omega})$. In this way, we recover the Smagorinsky model of \cite{cottet2003vorticity}.
\end{remark}

By a weak solution of \eqref{smagorinsky equation} we mean the following:
\begin{definition}\label{def well smag}
    We say that $\overline{\omega}$ is a weak solution of equation \eqref{smagorinsky equation} if \begin{align*}
        \overline{\omega}\in C_w(0,T;L^2(\mathbb{T}^2))\cap L^2(0,T;H^1(\mathbb{T}^2))
\end{align*}
and for each $\phi \in C^{\infty}(\mathbb{T}^2)$, for all $t\in [0,T]$, one has
\begin{align*}
    \langle \overline{\omega}_t,\phi\rangle-\langle \omega_0,\phi\rangle &= \nu \int_0^t \langle \overline{\omega}_s,\Delta \phi\rangle\, ds +\int_0^t \langle g(\overline{\omega}_s), \Delta\phi\rangle\, ds\\
    &\quad +\int_0^t \langle \overline{\omega}_s, K[\overline{\omega}_s]\cdot\nabla\phi \rangle\, ds .
\end{align*}
\end{definition}

In Section \ref{sect scaling limit} indeed we will first show the uniqueness of the weak solutions of \eqref{smagorinsky equation}, then we will show our main result which reads in the following way.

\begin{theorem}\label{Scaling Limit Smagorinski}
    Assume that $\{\theta^N\}_N \subset \ell^2$ satisfies \eqref{eq:theta} and \eqref{eq:theta-N}. Let $\omega^N$ be a weak solution of \eqref{Ito equation} corresponding to $\theta^N$, and $Q^N$ its law on $C([0,T];H^-(\mathbb{T}^2))\cap L^2(0,T;H^{1-}(\mathbb{T}^2))$. Then the family $\{Q^N \}_N$ is tight on $C([0,T];H^-(\mathbb{T}^2))\cap L^2(0,T;H^{1-}(\mathbb{T}^2))$ and it converges weakly to the Dirac measure $\delta_{\overline{\omega}}$, where $\overline{\omega}$ is the unique weak solution of equation \eqref{smagorinsky equation}.
\end{theorem}

\subsection{Preparatory results}

Before starting, we need to recall some results that we will use in Sections \ref{sec well posed} and \ref{sect scaling limit} in order to prove Theorems \ref{thm well posed} and \ref{Scaling Limit Smagorinski}, see \cite{simon1986compact, MarioMarco} for more details on these results.

In the following $X,\ B,\ Y$ are separable Banach spaces such that
\begin{align*}
    X\stackrel{c}{\hookrightarrow}B\hookrightarrow Y,
\end{align*}
where $\stackrel{c}{\hookrightarrow}$ means compact embedding.

\begin{theorem}\label{compactness 1}
    Let $p,r \in [1,+\infty]$ and $s\in \mathbb R$; assume that $s>0$ if $r\geq p$ or $s>1/r-1/p$ if $r\leq p$. Let $F$ be a bounded subset in $L^p(0,T;X)\cap W^{s,r}(0,T;Y)$. Then $F$ is relatively compact in $L^p(0,T;B)$ (in $C([0,T];B)$ if $p=+\infty$).
\end{theorem}
\begin{theorem}\label{compactness 2}
    Assume that there exists $\theta \in (0,1)$ such that
    \begin{align*}
        \lVert v\rVert_{B}\leq M\lVert v\rVert_X^{1-\theta}\lVert v\rVert_Y^{\theta} \quad \forall\, v\in X .
    \end{align*}
    Let $F$ be bounded in $W^{s_0,r_0}(0,T;X)\cap W^{s_1,r_1}(0,T;Y),\ r_0,\ r_1\in [1,+\infty]$. Define
      $$ s_{\theta}=(1-\theta)s_0+\theta s_1,\quad \frac{1}{r_{\theta}}=\frac{1-\theta}{r_0}+\frac{\theta}{r_1},\quad s_\ast=s_{\theta}-\frac{1}{r_{\theta}}.$$
    If $s_\ast<0$ then $F$ is relatively compact in $L^p(0,T;B)$ for each $p<-1/s_\ast$, and if $s_\ast>0$ then $F$ is relatively compact in $C([0,T];B)$.
\end{theorem}

\begin{lemma} \label{lemma Marco Mario}
Let $( \Omega, \mathcal{ A}, \mathbb{ P})$ be a probability space,
$\mathcal{U}$ and $\mathcal{H}$ separable Hilbert spaces.
Assume $ W= \sum_{k\geq 0} W_k e_k$ is an $(\mathcal{ F}_t)_{t\in[0,T]}$ cylindrical Brownian motion (over $\mathcal{U}$), while $ W^n= \sum_{k\ge 0} W^n_k e_k$ are $(\mathcal{ F}^n_t)_{t\in[0,T]}$ cylindrical Brownian motions (over $\mathcal{U}$). Assume that $ G$ is an $(\mathcal{ F}_t)_{t\in[0,T]}$ progressively measurable process which belongs to $L^2([0,T], L_2(\mathcal{U}, \mathcal H))$ $ \mathbb{P}$-a.s., while $ G^n$ are $(\mathcal{ F}^n_t)_{t\in[0,T]}$ progressively measurable processes which belong to $L^2([0,T], L_2(\mathcal{U}, \mathcal H))$ $ \mathbb{P}$-a.s.. If
\begin{subequations}\label{eq:ConvAsmp}
\begin{gather}
   W^{n}_k \rightarrow  W_k  \quad
  \textrm{ in probability in }
  C([0,T], \mathbb{R})\quad \forall k\ge 0,
  \label{eq:ConvAsmpNoise}\\
   G^{n} \rightarrow  G \quad
  \textrm{ in probability in }
  L^2([0,T]; L_2(\mathcal{U},  \mathcal H)),
  \label{eq:ConvAsmpStochIntegrand}
\end{gather}
\end{subequations}
then
\begin{equation}\label{eq:StochasticVitaliConc}
\sup_{t\in[0,T]}\left\Vert \int_0^t  G^n\, d W^n - \int_0^t  G\, d W \right\Vert_{\mathcal H} \to 0
\quad \textrm{ in probability}.
\end{equation}
\end{lemma}

In order to identify our limits we will use the following lemma on interpolation spaces.

\begin{lemma}\label{Lemma interpolation}
    Let $\chi_n,\chi\in L^{\infty}(0,T;H^-(\mathbb{T}^2))\cap L^2(0,T;H^{1-}(\mathbb{T}^2))$ such that \begin{align}\label{convergence}
        \chi_n\rightarrow \chi \quad \text{in } L^{\infty}(0,T;H^-(\mathbb{T}^2))\cap L^2(0,T;H^{1-}(\mathbb{T}^2)).
    \end{align}
    Then $\forall\, \beta> 2,\ \gamma\in [0,1)$ such that $\beta\gamma<2$
    \begin{align*}
    \chi_n\rightarrow \chi \in L^{\beta}(0,T;H^{\gamma}(\mathbb{T}^2)).
    \end{align*}
\end{lemma}
\begin{proof}
    Let $\delta,\ \delta'>0$ such that \begin{align}\label{assumption deltas}
        1-\delta>\gamma, \quad 2-2\delta-\beta\gamma>0, \quad \delta'<\frac{2-2\delta-\beta\gamma}{\beta-2}.
    \end{align} From our assumptions $\chi_n\rightarrow \chi \in L^{\infty}(0,T;H^{-\delta'}(\mathbb{T}^2))\cap L^{2}(0,T;H^{1-\delta}(\mathbb{T}^2))$. Then the thesis follows by interpolation inequalities and H\"older inequality. Indeed, it holds
    \begin{align*}
        \int_0^T\! \lVert \chi_n(t)-\chi(t)\rVert_{H^{\gamma}}^\beta dt
        & \leq \int_0^T\! \lVert \chi_n(t)-\chi(t)\rVert_{H^{1-\delta}}^{\beta\frac{\gamma+\delta'}{1-\delta+\delta'}} \lVert \chi_n(t)-\chi(t)\rVert_{H^{-\delta'}}^{\beta\frac{1-\gamma-\delta}{1-\delta+\delta'}} dt\\
        & \leq \lVert \chi_n- \chi\rVert_{L^{\infty}_t H^{-\delta'}_x}^{\beta\frac{1-\gamma-\delta}{1-\delta+\delta'}} \int_0^T\! \lVert \chi_n(t)-\chi(t)\rVert_{H^{1-\delta}}^{\beta\frac{\gamma+\delta'}{1-\delta+\delta'}} dt\\
        & \lesssim \lVert \chi_n- \chi\rVert_{L^{\infty}_t H^{-\delta'}_x}^{\beta\frac{1-\gamma-\delta}{1-\delta+\delta'}} \int_0^T\! \lVert \chi_n(t)\rVert_{H^{1-}}^{\beta\frac{\gamma+\delta'}{1-\delta+\delta'}} \!+ \! \lVert \chi(t)\rVert_{H^{1-}}^{\beta\frac{\gamma+\delta'}{1-\delta+\delta'}} dt,
    \end{align*}
    where $\|\cdot \|_{L^{\infty}_t H^{-\delta'}_x}$ is the norm in $L^\infty(0,T;H^{-\delta'})$. Under our assumptions on $\delta,\ \delta'$ it follows that $\beta\frac{\gamma+\delta'}{1-\delta+\delta'}\leq 2$. Therefore we have the thesis thanks to relation \eqref{convergence}.
\end{proof}

\section{Existence of solutions}\label{sec well posed}

Our approach for showing the existence of martingale solutions of system \eqref{Ito equation} follows by a standard compactness argument. See for example \cite[Section 2.4]{FlaLuoWaseda} and the references therein for some discussions on this method and further examples of application.

\subsection{Galerkin Approximation}

We introduce a sequence of Galerkin approximations $\omega^n$. Given the orthogonal projector $\Pi^n:L^2(\mathbb{T}^2)\rightarrow {\rm span}\{e_l,\ \lvert l\rvert\leq n\}$, we look for
\begin{align*}
    \omega^n(t)=\sum_{\lvert l \rvert \leq n} c_l(t)\, e_l
\end{align*}
such that $\forall\, \phi\in \Pi^n(L^2(\mathbb{T}^2))$, $\mathbb{P}$-a.s. $\forall\, t\in [0,T]$, it holds
\begin{align*}
    \langle \omega^n_t,\phi\rangle&= \langle \omega^n_0,\phi\rangle + \nu \int_0^t \langle \omega^n_s,\Delta \phi\rangle \, ds +\int_0^t \langle g(\omega^n_s), \Delta\phi\rangle\, ds\\
    &\quad +\int_0^t \langle K[\omega^n_s]\cdot\nabla\phi,\omega^n_s\rangle\, ds + \sum_{k\in \mathbb{Z}^2_0}\int_0^t \theta_k \langle \sigma_k\cdot\nabla \phi,f(\omega^n_s) \rangle\, dW^k_s,
\end{align*}
where $\omega_0^n=\Pi^n\omega_0$. Local existence of the solution $\omega^n$ is a classical fact due to the regularity of the coefficients appearing in the equation, see for example \cite{karatzas1991brownian, skorokhod1982studies}. Global existence follows from the following a priori estimates.

\begin{lemma}\label{a priori estimate Galerkin}
$\mathbb{P}$-a.s., $\omega^n$ satisfies
\begin{align}
  \lVert \omega^n_t\rVert^2+2\nu\int_0^t \lVert \nabla \omega^n_s\rVert^2\, ds \leq  \lVert \omega^n_0\rVert^2 &\leq \lVert \omega_0\rVert^2 \label{a priori 1}.
\end{align}

\end{lemma}

\begin{proof}
By It\^o formula and recalling the definition of $g$ we have \begin{align*}
        d\lVert \omega^n\rVert^2+2\nu \lVert \nabla \omega^n\rVert^2\, dt
        &= -2\langle K[\omega^n]\cdot\nabla \omega^n, \omega^n\rangle\, dt - \frac12 \langle f'(\omega^n)^2 \nabla \omega^n, \nabla\omega^n\rangle\, dt\\
        &\quad -2 \sum_{k} \theta_k\langle \sigma_k\cdot\nabla\omega^n,f(\omega^n)\rangle\, dW^k\\
        &\quad + \sum_{k} \theta_k^2 \lVert \Pi^n (\sigma_k \cdot \nabla f(\omega^n))\rVert^2\, dt.
    \end{align*}
    The first and the third terms are identically equal to $0$ due to the classical properties of the trilinear form of Navier-Stokes equations and the following relation:
    \begin{equation}\label{conservation of enstrophy}
        \langle \sigma_k\cdot\nabla\omega^n,f(\omega^n)\rangle =\langle \sigma_k,\nabla F(\omega^n)\rangle \notag= -\langle\operatorname{div}\sigma_k,F(\omega^n)\rangle=0,
    \end{equation} where the function $F$ above is a primitive of $f.$
    Therefore we are left to show that \begin{align*}
        -\frac12 \langle f'(\omega^n)^2 \nabla \omega^n, \nabla\omega^n\rangle +\sum_{k} \theta_k^2 \lVert \Pi^n (\sigma_k \cdot \nabla f(\omega^n))\rVert^2\leq 0.
    \end{align*}
    The last inequality is due to
    \begin{align*}
        \sum_{k} \theta_k^2 \lVert \Pi^n (\sigma_k \cdot \nabla f(\omega^n))\rVert^2
        & \leq \sum_{k} \theta_k^2 \lVert \sigma_k \cdot \nabla f(\omega^n)\rVert^2\\
        & = \sum_{k} \theta_k^2 \int_{\mathbb T^2} (\nabla f(\omega^n))^\ast (\sigma_k\otimes \sigma_k) \nabla f(\omega^n)\, dx \\
        & = \frac12\lVert \nabla f(\omega^n)\rVert^2 = \frac12 \langle f'(\omega^n)^2 \nabla \omega^n, \nabla\omega^n\rangle,
    \end{align*}
where in the third step we have used \eqref{eq:covariance-diagonal}.
\end{proof}

Lemma \ref{a priori estimate Galerkin} shows in particular that $\{\omega^n\}_{n\geq 1}$ is bounded in $L^p(\Omega;L^p(0,T;L^2))\cap L^2(\Omega;L^2(0,T;H^1)).$ In order to apply Theorem \ref{compactness 1} and Theorem \ref{compactness 2} we need some energy estimates in $W^{s,r}(0,T;H^{-\beta})$, $s\geq 0, r\geq 2, \beta> 0 $ satisfying suitable conditions. To this end we first prove the following Lemma.

\begin{lemma}\label{solution vs eigenfunctions}
    For each $M\in \mathbb{N}$, there exists a constant $C$ independent of $n$ such that for all $ 0\leq s\leq t\leq T$ it holds
    \begin{align*}
        \mathbb{E}\left[\langle \omega^n_t-\omega^n_s,e_l\rangle^M\right]& \leq C \left(1+\lVert \omega_0\rVert^{M(2\vee (2\alpha+1))}\right)\lvert l\rvert^{2M}\lvert t-s\rvert^{M/2},
    \end{align*}
    where $\alpha\in [0,1]$ is the parameter in \eqref{bound f''}.
\end{lemma}

\begin{proof}
    It is enough to consider $\lvert l \rvert\leq n $. From the weak formulation satisfied by $\omega^n$ it follows that
    \begin{align*}
         \langle \omega^n_t-\omega^n_s,e_l\rangle&= \nu \int_s^t \langle \omega^n_r,\Delta e_l\rangle \, dr + \int_s^t \langle g(\omega^n_r), \Delta e_l\rangle\, dr\\
         &\quad +\int_s^t \langle K[\omega^n_r]\cdot\nabla e_l,\omega^n_r\rangle\, dr + \sum_{k} \theta_k \langle \sigma_k\cdot \nabla e_l,f(\omega^n_r) \rangle\, dW^k_r\\
         & =I^1_{s,t}+I^2_{s,t}+I^3_{s,t}+I^4_{s,t}.
    \end{align*}
    The analysis of $I^1_{s,t}$ and $I^3_{s,t}$ follows arguing exactly as in \cite[Lemma 3.4]{flandoli2021scaling} and leads us to
    \begin{align*}
\mathbb{E}\left[(I^1_{s,t})^M\right]+\mathbb{E}\left[(I^3_{s,t})^M\right]\lesssim \lVert \omega_0\rVert^{M}\lvert l\rvert^{2M} \lvert t-s\rvert^M+\lVert \omega_0\rVert^{2M}\lvert l\rvert^M \lvert t-s\rvert^M.
    \end{align*}
For what concerns $I_{s,t}^2$ with $\alpha\in [1/2,1]$ (the case $\alpha\in [0,1/2]$ being easier), we have by H\"older's inequality and relation \eqref{bound g} that
  $$\aligned
  \mathbb{E}\left[(I^2_{s,t})^M\right]
  &\lesssim \mathbb{E}\left[\left \lvert \int_s^t \langle  g(\omega^n_r), \Delta e_l\rangle\, dr\right\rvert^M\right]\\
  & \leq \mathbb{E}\left[\left \lvert \int_s^t \lVert g(\omega^n_r)\rVert_{L^1} \lVert \Delta e_l\rVert_{L^{\infty}} \, dr\right\rvert^M\right]\\
  & \lesssim \lvert l\rvert^{2M}\, \mathbb{E}\left[\left \lvert \int_s^t \lVert 1+\lvert\omega^n_r \rvert^{2\alpha+1}\rVert_{L^1}\, dr\right\rvert^M\right]\\
  & \lesssim \lvert l\rvert^{2M}\left(\lvert t-s\rvert^M +\mathbb{E}\left[\left \lvert \int_s^t \lVert \omega^n_r\rVert_{L^{2\alpha+1}}^{2\alpha+1}\, dr\right\rvert^M\right] \right).
  \endaligned$$
Next, by Sobolev embedding theorem and interpolation inequalities,
\begin{align*}
    \mathbb{E}\left[(I^2_{s,t})^M\right]
    & \lesssim \lvert l\rvert^{2M}\left(\lvert t-s\rvert^M +\mathbb{E}\left[\left \lvert \int_s^t \lVert \omega^n_r\rVert_{H^{\frac{2\alpha-1}{2\alpha+1}}}^{2\alpha+1} dr\right\rvert^M\right]\right)\\
    & \leq \lvert l\rvert^{2M}\left(\lvert t-s\rvert^M +\mathbb{E}\left[\left \lvert \int_s^t \lVert \nabla\omega^n_r\rVert^{2\alpha-1} \lVert \omega_r^n\rVert^2 dr\right\rvert^M\right]\right),
\end{align*}
which, combined the estimates in Lemma \ref{a priori estimate Galerkin}, yields
\begin{align*}
    \mathbb{E}\left[(I^2_{s,t})^M\right] & \leq \lvert l\rvert^{2M}\left(\lvert t-s\rvert^M +\lvert t-s\rvert ^{M(\frac32-\alpha)}\lVert \omega_0\rVert^{2M}\mathbb{E}\left[\left \lvert \int_s^t \lVert \nabla\omega^n_r\rVert^{2} dr\right\rvert^{M(\alpha-\frac12)} \right]\right)\\
    & \lesssim \lvert l\rvert^{2M}\left(\lvert t-s\rvert^M +\lvert t-s\rvert ^{M(\frac32-\alpha)} \lVert \omega_0\rVert^{M(2\alpha+1)}\right)\\
    & \lesssim \lvert l \rvert^{2M}(1+\lVert \omega_0\rVert^{M(2\alpha+1)})\lvert t-s\rvert^{M\left(1\wedge (\frac32- \alpha) \right)}.
\end{align*}

Lastly we need to deal with $I_{s,t}^4$. Recall that $\theta\in \ell^2(\mathbb Z^2_0)$ fulfills $\|\theta \|_{\ell^2}=1$, and $\|\sigma_k\|_{L^\infty} =\sqrt{2}$; by Burkholder-Davis-Gundy inequality and estimate \eqref{bound f'},
\begin{align*}
    \mathbb{E}\left[(I^4_{s,t})^M\right]
    & \lesssim \mathbb{E}\left[\left\lvert \sum_{k} \theta_k^2 \int_s^t\langle \sigma_k\cdot\nabla e_l,f(\omega^n_r) \rangle^2\, dr\right\rvert^{M/2}\right]\\
    & \lesssim \mathbb{E}\left[\left\lvert \int_s^t\lVert f(\omega^n_r) \nabla e_l \rVert_{L^1}^2 \, dr\right\rvert^{M/2}\right]\\
    & \lesssim \lvert l\rvert^M\mathbb{E}\left[\left\lvert \int_s^t\lVert 1+\lvert \omega^n_r\rvert^{\alpha+1} \rVert_{L^1}^2\, dr\right\rvert^{M/2}\right] .
\end{align*}
Then, similarly as for the treatment of $I^2_{s,t}$, by Sobolev embedding theorem, interpolation inequalities and Lemma \ref{a priori estimate Galerkin} we have
\begin{align*}
    \mathbb{E}\left[(I^4_{s,t})^M\right]
    & \lesssim \lvert l\rvert^M\left(\lvert t-s\rvert^{M/2}+\mathbb{E}\left[\left\lvert \int_s^t\lVert \omega^n_r \rVert_{L^{\alpha+1}}^{2(\alpha+1)}\, dr\right\rvert^{M/2}\right]\right)\\
    & \lesssim \lvert l\rvert^M\left(\lvert t-s\rvert^{M/2}+\mathbb{E}\left[\left\lvert \int_s^t\lVert \omega^n_r \rVert^{2(\alpha+1)}\, dr\right\rvert^{M/2}\right]\right)\\
    & \lesssim \lvert l\rvert^M\lvert t-s\rvert^{M/2}\left(1+\lVert \omega_0\rVert^{M(\alpha+1)}\right).
\end{align*}
Combining the estimates the thesis follows.
\end{proof}

By Theorem \ref{compactness 1}, a set bounded in $L^2(0,T;H^{1})\cap W^{s,r}(0,T;H^{-\gamma})$ is relatively compact in $L^2(0,T;H^{1-\delta})$ for each $\delta>0$ if $s>0, \gamma>0, r\geq 2.$ On the other side, given $\delta>0$, if $p>\frac{r_1}{\delta(s_1r_1-1)(\beta-\delta)}$, a set bounded in $L^p(0,T;L^2)\cap W^{s_1,r_1}(0,T; H^{-\beta})$ with $s_1r_1>1$ is relatively compact in $C(0,T;H^{-\delta})$. Since by Lemma \ref{a priori estimate Galerkin} we can take $p$ arbitrarily large, it is enough to show the boundedness of $\{\omega^n \}_n$ in $W^{s_1,r_1}(0,T; H^{-\beta})$ for some $\beta.$ This is guaranteed by the lemma below.

\begin{lemma}\label{Lemma tightness galerkin}
    If $\beta>3+\frac{2}{r_1}, \ s_1<\frac{1}{2}, \ s_1r_1>1$  there exists a constant $C$ independent of $n$ such that
    \begin{align*}
        \mathbb{E}\left[\int_0^T\lVert \nabla\omega^n_s\rVert^2 ds \right]+\mathbb{E}\left[\int_0^T\lVert \omega^n_s\rVert^p ds \right]+\mathbb{E}\left[\int_0^T dt \int_0^T ds \frac{\lVert \omega^n_t-\omega^n_s\rVert_{H^{-\beta}}^{r_1}}{|t-s|^{1+r_1s_1}}\right]\leq C.
    \end{align*}
\end{lemma}

\begin{proof}
    Thanks to Lemma \ref{a priori estimate Galerkin} we need just to consider \begin{align*}
        \mathbb{E}\left[\int_0^T dt \int_0^T ds \frac{\lVert \omega^n_t-\omega^n_s\rVert_{H^{-\beta}}^{r_1}}{|t-s|^{1+r_1s_1}}\right].
    \end{align*}
    By Fubini theorem it follows that
    \begin{align*}
         \mathbb{E}\left[\int_0^T dt \int_0^T ds \frac{\lVert \omega^n_t-\omega^n_s\rVert_{H^{-\beta}}^{r_1}}{|t-s|^{1+r_1s_1}}\right]& = \int_0^T dt \int_0^T ds \frac{\mathbb{E}\left[\lVert \omega^n_t-\omega^n_s\rVert_{H^{-\beta}}^{r_1}\right]}{|t-s|^{1+r_1s_1}}.
    \end{align*}
    Let us understand better $\mathbb{E}\left[\lVert \omega^n_t-\omega^n_s\rVert_{H^{-\beta}}^{r_1}\right]$: by the definition of Sobolev norms and H\"older's inequality,
    \begin{align*}
        \mathbb{E}\left[\lVert \omega^n_t-\omega^n_s\rVert_{H^{-\beta}}^{r_1}\right]
        &=\mathbb{E}\left[\left(\sum_{l\in\mathbb{Z}^2_0}\frac{\langle \omega^n_t-\omega^n_s, e_l \rangle^2}{\lvert l\rvert^{2\beta}}\right)^{r_1/2}\right]\\
        & =\mathbb{E}\left[\left(\sum_{l\in\mathbb{Z}^2_0}\frac{\langle \omega^n_t-\omega^n_s, e_l \rangle^2}{\lvert l\rvert^{2(\beta-\frac{(1+\epsilon)(r_1-2)}{r_1})}\lvert l\rvert^{2(\frac{(1+\epsilon)(r_1-2)}{r_1})}} \right)^{r_1/2}\right]\\
        & \leq \left(\sum_{l\in\mathbb{Z}^2_0}\frac{1}{\lvert l\rvert^{2(1+\epsilon)}}\right)^{(r_1-2)/2}\sum_{l\in\mathbb{Z}^2_0}\mathbb{E}\left[\frac{\langle \omega^n_t-\omega^n_s, e_l \rangle^{r_1}}{\lvert l\rvert^{\beta r_1-(1+\epsilon)(r_1-2)}}\right] .
    \end{align*}
    Thanks to Lemma \ref{solution vs eigenfunctions}, we have
    \begin{align*}
        \mathbb{E}\left[\lVert \omega^n_t-\omega^n_s\rVert_{H^{-\beta}}^{r_1}\right]
        & \lesssim \sum_{l\in\mathbb{Z}^2_0}\frac{\left(1+\lVert \omega_0\rVert^{r_1(2\vee (2\alpha+1))}\right)\lvert l\rvert^{2r_1}\lvert t-s\rvert^{r_1/2}}{\lvert l\rvert^{\beta r_1-(1+\epsilon)(r_1-2)}}\\
        & \lesssim \lvert t-s\rvert^{r_1/2} \sum_{l\in\mathbb{Z}^2_0}\frac{1}{\lvert l\rvert^{r_1(\beta-3)}}\\
        & \lesssim  \lvert t-s\rvert^{r_1/2}.
    \end{align*}
    Therefore \begin{align*}
        \mathbb{E}\left[\int_0^T dt \int_0^T ds \frac{\lVert \omega^n_t-\omega^n_s\rVert_{H^{-\beta}}^{r_1}}{|t-s|^{1+r_1s_1}}\right]&\lesssim \int_0^T dt \int_0^T ds \frac{1}{\lvert t-s\rvert^{1+r_1(s_1-1/2)}} \lesssim 1.
    \end{align*}
    The proof is complete.
\end{proof}

Combining Lemma \ref{Lemma tightness galerkin} with Theorems \ref{compactness 1}, \ref{compactness 2} we have the following tightness result by Markov's inequality:

\begin{corollary}
The family of laws of $\omega^n$ is tight on $C([0,T];H^-)\cap L^2(0,T;H^{1-}).$
\end{corollary}

\subsection{Passage to the limit}\label{passage to the limit Galerkin}

Arguing as in \cite{flandoli2021scaling}, by Skorohod's representation Theorem, we can find, up to passing to subsequences, an auxiliary probability space, that for simplicity we continue to call $(\Omega,\mathcal{F},\mathbb{P})$, and processes $(\Tilde{\omega}^n,W^n:=\{W^{n,k}\}_{k\in \mathbb{Z}^2_0}),\  ({\omega},W:=\{W^{k}\}_{k\in \mathbb{Z}^2_0}),  $ such that \begin{align*}
    &\Tilde{\omega}^n\rightarrow \omega \quad \text{in } C([0,T];H^-)\cap L^2(0,T;H^{1-}) \quad \mathbb{P}-a.s.\\
   & W^n\rightarrow W \quad \text{in } C([0,T];\mathbb{Z}^2_0)\quad \mathbb{P}-a.s.
\end{align*}
Of course the convergence above between $W^n$ and $W$ can be seen as the uniform convergence of cylindrical Wiener processes $W^n=\sum_{k\in \mathbb{Z}^2_0}e_k W^{n,k},\ W=\sum_{k\in \mathbb{Z}^2_0}e_k W^{k}$ on a suitable Hilbert space $U_0$. Before going on, in order to identify $\omega$ as a weak solution of equation \eqref{Ito equation} we need further integrability properties of $\omega$. The proof of the proposition below is analogous to Lemma 3.5 in \cite{flandoli2021scaling}, therefore we will omit the details in these notes.

\begin{proposition}\label{further estimates solution}
The process $\omega$ has weakly continuous trajectories on $L^2(\mathbb{T}^2)$ and satisfies
\begin{align*}
    \operatorname{sup}_{t\in [0,T]}\lVert \omega_t\rVert^2\leq \lVert \omega_0\rVert^2 \quad\mathbb{P}-a.s.\\
    2\nu \int_0^T \lVert \nabla \omega_s\rVert^2 ds \leq \lVert \omega_0\rVert^2 \quad\mathbb{P}-a.s.
\end{align*}
\end{proposition}

Now we are ready to prove Theorem \ref{thm well posed}.

\begin{proof}[Proof of Theorem \ref{thm well posed}]
Let $\phi\in \Pi^M(L^2(\mathbb{T}^2))$, by classical arguments for each $n\geq M$, $\Tilde{\omega}^n$ satisfies the following weak formulation: $\mathbb{P}$-a.s. for all $t\in [0,T]$,
\begin{align*}
    \langle \Tilde{\omega}^n_t-\omega^n_0,\phi\rangle
    &= \nu \int_0^t \langle \Tilde{\omega}^n_s,\Delta \phi\rangle\, ds+ \int_0^t \langle  g(\Tilde{\omega}^n_s), \Delta\phi\rangle\, ds\\
    &\quad +\int_0^t \langle \Tilde{\omega}^n_s, K[\tilde{\omega}^n_s]\cdot\nabla\phi \rangle\, ds + \sum_{k\in \mathbb{Z}^2_0}\int_0^t \theta_k \langle f(\Tilde{\omega}^n_s), \sigma_k\cdot\nabla \phi \rangle\, dW^{n,k}_s .
\end{align*}
    Therefore we will show, up to passing to a further subsequence, $\mathbb{P}$-a.s. convergence of all the terms appearing above, uniformly in time. Indeed,
    \begin{align}\label{convergence time t}
        \operatorname{sup}_{t\in [0,T]}\lvert \langle\Tilde{\omega}^n_t-\omega_t,\phi\rangle\rvert\leq \lVert \Tilde{\omega}^n-\omega\rVert_{C([0,T];H^-)}\lVert \phi\rVert_{H^1}\rightarrow 0 \quad \mathbb{P} \mbox{-a.s.}
    \end{align}
    and similarly for the initial conditions. Next,
    \begin{align}\label{convergence laplacian}
    \operatorname{sup}_{t\in [0,T]}\left\lvert \int_0^t \langle \Tilde{ \omega}^n_s-\omega_s,\Delta \phi\rangle \, ds\right\rvert & \leq \lVert \phi \rVert_{H^2}\int_0^T\lVert \omega_s-\Tilde{\omega}^n_s\rVert \, ds\rightarrow 0\quad
     \mathbb{P} \mbox{-a.s.}
    \end{align}
    due to the almost surely convergence in $L^2(0,T;H^{1-})$. Moreover,
    \begin{align}\label{convergence nonlinear}
        &\operatorname{sup}_{t\in [0,T]}\left\lvert \int_0^t \langle \Tilde{ \omega}^n_s, K[\tilde{\omega}^n_s]\cdot\nabla \phi\rangle -\langle \omega_s, K[\omega_s]\cdot\nabla \phi\rangle\, ds\right\rvert\notag\\
        & \leq \int_0^T \lvert\langle \Tilde{ \omega}^n_s, (K[\tilde{\omega}^n_s]-K[\omega_s])\cdot\nabla \phi\rangle\rvert\, ds +\int_0^T \lvert\langle \Tilde{ \omega}^n_s-\omega_s, K[\omega_s]\cdot\nabla \phi\rangle\rvert \,ds\notag\\
        & \lesssim \lVert \phi\rVert_{W^{1,\infty}}\lVert \omega_0\rVert\int_0^T \lVert \Tilde{\omega}^n_s-\omega_s\rVert ds\rightarrow 0 \quad \mathbb{P} \mbox{-a.s.}
    \end{align}
    due to the almost surely convergence in $L^2(0,T;H^{1-})$. Thanks to relation \eqref{bound g} it follows that
    \begin{align}\label{convergence smagorinsky diffusion step 1}
         &\operatorname{sup}_{t\in [0,T]}\left\lvert \int_0^t \langle g(\Tilde{\omega}^n_s)-g(\omega_s), \Delta \phi\rangle \, ds\right\rvert\notag\\
         & \leq \lVert \Delta \phi\rVert_{L^{\infty}}\int_0^T  \left( \left\lVert \left\lvert \tilde{\omega}^n_s-\omega_s\right\rvert \big(1+ |\omega_s|^{2\alpha} \big) \right\rVert_{L^1} +\left\lVert \left\lvert \tilde{\omega}^n_s\right\rvert \big| | \tilde{\omega}^n_s|^{2\alpha}- | \omega_s|^{2\alpha}\big| \right\rVert_{L^1}\right) ds \notag \\
         & =  \lVert \phi\rVert_{W^{2,\infty}}(I_1+I_2),
    \end{align}
    where \begin{align*}
        I_1&=\int_0^T  \left\lVert \left\lvert \tilde{\omega}^n_s-\omega_s\right\rvert \big(1+ |\omega_s|^{2\alpha} \big) \right\rVert_{L^1}\, ds, \\
        I_2&= \int_0^T  \big\lVert \left\lvert \tilde{\omega}^n_s\right\rvert  \left( \left\lvert\tilde{\omega}^n_s\right\rvert^{\alpha}-\left\lvert \omega_s\right\rvert^{\alpha} \right)\left( \left\lvert\tilde{\omega}^n_s\right\rvert^{\alpha}+\left\lvert \omega_s\right\rvert^{\alpha} \right) \big\rVert_{L^1}\, ds.
    \end{align*}
    Let us show that, $\mathbb{P}$-a.s., both $I_1$ and $I_2$ tend to $0$.
    We can control $I_1$ thanks to H\"older inequality, Sobolev embedding theorem, interpolation inequalities,
    \begin{align}
        I_1 & \leq\int_0^T \lVert \tilde{\omega}^n_s-\omega_s\rVert \left(1+\lVert \omega_s\rVert_{L^{4\alpha}}^{2\alpha}\right) ds \notag\\
        & \leq \int_0^T \lVert \tilde{\omega}^n_s-\omega_s\rVert \left(1+\lVert \omega_s\rVert_{H^{\frac{2\alpha-1}{2\alpha}}}^{2\alpha}\right) ds\notag\\
        & \leq \int_0^T \lVert \tilde{\omega}^n_s-\omega_s\rVert \left(1+\lVert \omega_s\rVert_{H^1}^{2\alpha-1}\lVert \omega_s\rVert\right) ds \notag .
    \end{align}
    By Lemma \ref{Lemma interpolation}, we have for $\alpha\in (1/2,1]$ (the other case being easier) that
    \begin{align}\label{convergence smagorinsky diffusion step 2}
        I_1 & \lesssim \lVert \tilde{\omega}^n-\omega\rVert_{L^2(0,T;L^2)}  +\lVert \omega_0\rVert\, \lVert \omega\rVert_{L^2(0,T;H^1)}^{2\alpha-1}\lVert \tilde{\omega}^n-\omega\rVert_{L^{\frac{2}{3-2\alpha}}(0,T;L^2)}\notag\\
        & \lesssim \lVert \tilde{\omega}^n-\omega\rVert_{L^2(0,T;L^2)}  +\lVert \omega_0\rVert^{2\alpha} \lVert \tilde{\omega}^n-\omega\rVert_{L^{\frac{2}{3-2\alpha}}(0,T;L^2)} \stackrel{\mathbb{P} \mbox{\footnotesize -a.s.}}{\longrightarrow } 0 .
    \end{align}
For what concerns $I_2$ similar arguments and the H\"olderianity of $x^{\alpha}$ lead to
\begin{align}
    I_2& \leq \int_0^T \lVert \tilde{\omega}^n_s\rVert \left\lVert \left\lvert\tilde{\omega}^n_s-\omega_s \right\rvert^{\alpha}\left(\left\lvert\tilde{\omega}^n_s\right\rvert^{\alpha}+\left\lvert \omega_s\right\rvert^{\alpha}\right)  \right\rVert \notag\\
    & \lesssim \lVert \omega_0\rVert \int_0^T \lVert  \tilde{\omega}^n_s-\omega_s \rVert_{L^{4\alpha}}^{\alpha}\left(\lVert  \tilde{\omega}^n_s \rVert_{L^{4\alpha}}^{\alpha}+\lVert  \omega_s \rVert_{L^{4\alpha}}^{\alpha}\right) ds \notag \\
    & \lesssim \lVert \omega_0\rVert \int_0^T \lVert  \tilde{\omega}^n_s-\omega_s \rVert_{H^{\frac{2\alpha-1}{2\alpha}}}^{\alpha}\left(\lVert  \tilde{\omega}^n_s \rVert_{H^{\frac{2\alpha-1}{2\alpha}}}^{\alpha}+\lVert  \omega_s \rVert_{H^{\frac{2\alpha-1}{2\alpha}}}^{\alpha}\right) ds \notag .
\end{align}
By interpolation and H\"older's inequality,
\begin{align}\label{convergence smagorinsky diffusion step 3}
    I_2 & \leq \lVert \omega_0\rVert \int_0^T \lVert  \tilde{\omega}^n_s-\omega_s \rVert_{H^{\frac{2\alpha-1}{2\alpha}}}^{\alpha}\left(\lVert  \tilde{\omega}^n_s \rVert_{H^{1}}^{\frac{2\alpha-1}{2}}\lVert \tilde{\omega}^n_s\rVert^{\frac{1}{2}} +\lVert  \omega_s \rVert_{H^{1}}^{\frac{2\alpha-1}{2}}\lVert \omega_s\rVert^{\frac{1}{2}}\right) ds \notag\\
    & \leq \lVert \omega_0\rVert^{3/2}\left(\lVert\tilde\omega^n\rVert_{L^2(0,T;H^1)}^{\frac{2\alpha-1}{2}}+\lVert \omega\rVert_{L^2(0,T;H^1)}^{\frac{2\alpha-1}{2}} \right) \lVert \tilde{\omega}^n_s-\omega_s\rVert_{L^{\frac{4\alpha}{5-2\alpha}}(0,T;H^{\frac{2\alpha-1}{2\alpha}})}^{\alpha}\notag\\
    & \lesssim \lVert \omega_0\rVert^{1+\alpha}\lVert \tilde{\omega}^n_s-\omega_s\rVert_{L^{\frac{4\alpha}{5-2\alpha}}(0,T;H^{\frac{2\alpha-1}{2\alpha}})}^{\alpha} \stackrel{\mathbb{P} \mbox{\footnotesize -a.s.}}{\longrightarrow } 0 .
\end{align}

    In order to deal with the stochastic integral we apply Lemma \ref{lemma Marco Mario}. Since we have the convergence of the Wiener processes, it is enough to show that $\mathbb{P}$-a.s., therefore in probability,
    \begin{align*}
        \int_0^T \sum_{k} \theta_k^2 \langle \sigma_k\cdot\nabla \phi, f(\Tilde{\omega}^n_s)-f(\omega_s)\rangle^2\, ds \rightarrow 0.
    \end{align*}
    The relation above is true, indeed, recall the facts that $\lVert \sigma_k\rVert_{L^{\infty}}=\sqrt{2}\ (\forall\, k\in \mathbb{Z}^2_0)$, $\sum_{k\in \mathbb{Z}^2_0}\theta_k^2=1$, and relation \eqref{bound f''} we have
    \begin{align}
        &\int_0^T \sum_{k} \theta_k^2 \langle \sigma_k\cdot\nabla \phi, f(\Tilde{\omega}^n_s)-f(\omega_s)\rangle^2 ds\notag\\
        & \leq \lVert \phi\rVert_{W^{1,\infty}}^2\! \int_0^T \sum_{k} \theta_k^2 \lVert \sigma_k\rVert_{L^{\infty}}^2\lVert f(\Tilde{\omega}^n_s)-f(\omega_s)\rVert_{L^1}^2 ds\notag\\
        & \lesssim  \lVert \phi\rVert_{W^{1,\infty}}^2\!  \int_0^T \big\lVert \lvert \Tilde{\omega}^n_s-\omega_s\rvert +  | \tilde{\omega}^n_s  \left\lvert \tilde{\omega}^n_s\right\rvert^{\alpha}- \omega_s |\omega_s |^{\alpha} | \big\rVert_{L^1}^2 ds\notag\\
        & \lesssim \lVert \phi\rVert_{W^{1,\infty}}^2\!  \left(\lVert \Tilde{\omega}^n-\omega\rVert_{L^2_t L^2_x} +\! \int_0^T\! \big\lVert |\tilde\omega^n_s-\omega_s|\, |\tilde\omega^n_s|^{\alpha}  \big\rVert_{L^1}^2 + \big\lVert |\omega_s| \, |\tilde\omega^n_s-\omega_s |^{\alpha} \big\rVert_{L^1}^2 ds\right) \notag \! ,
    \end{align}
    where $\|\cdot \|_{L^2_t L^2_x}$ is the norm in $L^2(0,T; L^2(\mathbb T^2))$. By Cauchy's inequality,
    \begin{align}
        &\int_0^T \sum_{k} \theta_k^2 \langle \sigma_k\cdot\nabla \phi, f(\Tilde{\omega}^n_s)-f(\omega_s)\rangle^2 ds\notag\\
        & \leq \lVert \phi\rVert_{W^{1,\infty}}^2\! \left(\lVert \Tilde{\omega}^n-\omega\rVert_{L^2_t L^2_x} +\! \int_0^T\!  \lVert \tilde\omega^n_s-\omega_s\rVert^{2}\lVert \tilde \omega_s^n\rVert_{L^{2\alpha}}^{2\alpha}+ \lVert \omega_s\rVert^{2}\lVert \tilde \omega_s^n-\omega_s\rVert_{L^{2\alpha}}^{2\alpha}  ds\right)\notag\\
        &  \leq \lVert \phi\rVert_{W^{1,\infty}}^2\! \left(\lVert \Tilde{\omega}^n-\omega\rVert_{L^2_t L^2_x}+\lVert \omega_0\rVert^{2\alpha}\lVert \Tilde{\omega}^n-\omega\rVert_{L^2_t L^2_x}^2+\lVert \omega_0\rVert^{2}\lVert \Tilde{\omega}^n-\omega\rVert_{L^2_t L^2_x}^{2\alpha}\right)\notag \\
        & \rightarrow 0 \quad \mathbb{P} \mbox{-a.s.}
    \end{align}
        Therefore by Lemma \ref{lemma Marco Mario}, up to passing to a subsequence, uniformly in time, \begin{align}\label{convergence martingale}
        \sum_{k}\int_0^t\theta_k \langle \sigma_k\cdot\nabla \phi,f(\Tilde{\omega}^n_s) \rangle\, dW^{n,k}_s \ \stackrel{\mathbb{P} \mbox{\footnotesize -a.s}}{\longrightarrow}\ \sum_{k}\int_0^t\theta_k \langle \sigma_k\cdot\nabla \phi,f(\omega_s) \rangle\, dW^{k}_s .
    \end{align}
    Combining relations \eqref{convergence time t}, \eqref{convergence laplacian}, \eqref{convergence nonlinear}, \eqref{convergence smagorinsky diffusion step 1}, \eqref{convergence smagorinsky diffusion step 2}, \eqref{convergence smagorinsky diffusion step 3}, \eqref{convergence martingale} we have, $\mathbb{P}$-a.s. for all $t\in [0,T]$,
    \begin{align}\label{final approximation galerkin}
    \langle \omega_t,\phi\rangle &=\langle \omega_0,\phi\rangle +
     \nu \int_0^t \langle \omega_s,\Delta \phi\rangle \, ds+\int_0^t \langle g(\omega_s), \Delta\phi\rangle\, ds\notag\\
    &\quad +\int_0^t \langle K[\omega_s]\cdot\nabla\phi,\omega_s\rangle\, ds + \int_0^t \sum_{k} \theta_k\langle \sigma_k\cdot\nabla \phi,f(\omega_s) \rangle\, dW^k_s .
    \end{align}
    By standard density argument we can find a zero measure set $N$ such that on its complementary relation \eqref{final approximation galerkin} holds for each $\phi\in C^{\infty}(\mathbb{T}^2)$.
\end{proof}

\section{Scaling limit}\label{sect scaling limit}

Let now $\{\theta^N \}_N$ be a sequence in $\ell^2(\mathbb Z^2_0)$, each satisfying the conditions \eqref{eq:theta} and moreover
  \begin{equation}\label{eq:theta-N-1}
  \lim_{N\rightarrow +\infty} \lVert \theta^N\rVert_{\ell^\infty}=0;
  \end{equation}
let  $\omega^N$ be an analytically weak martingale solution in the sense of Definition \ref{well posed def} of
\begin{equation}\label{Ito equation scaled}
    \begin{cases}
        d\omega^N=\big(\nu \Delta \omega^N- K[\omega^N]\cdot\nabla \omega^N +\Delta g(\omega^N) \big)\, dt+\sum_{k}\theta^{N}_k \sigma_k\cdot \nabla f(\omega^N) \,  dW^k, \\
        u^N=-\nabla^{\perp}(-\Delta)^{-1}\omega^N, \\
        \omega^N(0)=\omega_0
    \end{cases}
\end{equation}
satisfying
\begin{align*}
        \sup_{t\in [0,T]}\lVert \omega^N_t\rVert^2+2\nu\int_0^T \lVert \nabla \omega^N_s\rVert^2\, ds\leq 2\lVert \omega_0\rVert^2 \quad \mathbb{P} \mbox{-a.s.}
    \end{align*}
The existence of such solution for each $N\in\mathbb{N}$ is guaranteed by Theorem \ref{thm well posed} above. Of course the probability space and the Brownian motions depend from $N$, however with some abuse of notation, we do not stress this dependence. Arguing as in Section \ref{sec well posed} we will show the tightness of the law of $\omega^N$ in $C([0,T];H^-)\cap L^2(0,T;H^{1-})$. This will allow us to prove Theorem \ref{Scaling Limit Smagorinski} following the same ideas of Section \ref{sec well posed}.

\subsection{Tightness}
The way of showing the tightness is completely analogous to Section \ref{sec well posed} thanks to Proposition \ref{further estimates solution}. Therefore we just sketch the argument. We start with the lemma below.

\begin{lemma}\label{solution scaled vs eigenfunctions}
    For each $M\in \mathbb{N}$, there exists a constant $C$ independent of $N$ such that for any $s,t$ with $0\leq s\leq t\leq T$, it holds
    \begin{align*}
        \mathbb{E}\left[\langle \omega^N_t-\omega^N_s,e_l\rangle^M\right]& \leq C \left(1+\lVert \omega_0\rVert^{M(2\vee (2\alpha+1))}\right)\lvert l\rvert^{2M}\lvert t-s\rvert^{M/2}.
    \end{align*}
\end{lemma}

\begin{proof}
 From the weak formulation satisfied by $\omega^N$ it follows that
    \begin{align*}
         \langle \omega^N_t-\omega^N_s,e_l\rangle &= \nu \int_s^t \langle \omega^N_r,\Delta e_l\rangle \, dr +\int_s^t \langle  g(\omega^N_r), \Delta e_l\rangle\, dr\\
         &\quad +\int_s^t \langle K[\omega^N_r]\cdot\nabla e_l,\omega^N_r\rangle dr + \sum_{k}\theta^{N}_k \langle \sigma_k\cdot\nabla e_l,f(\omega^N_r) \rangle\, dW^k_r\\
         & =I^1_{s,t}+I^2_{s,t}+I^3_{s,t}+I^4_{s,t}.
    \end{align*}
    All the terms above can be treated analogously to Lemma \ref{solution vs eigenfunctions}, leading us to the following estimates:
    \begin{align*}
\mathbb{E}\left[(I^1_{s,t})^M\right]+\mathbb{E}\left[(I^3_{s,t})^M\right]&\lesssim \lVert \omega_0\rVert^{M}\lvert l\rvert^{2M} \lvert t-s\rvert^M+\lVert \omega_0\rVert^{2M}\lvert l\rvert^M \lvert t-s\rvert^M,\\
    \mathbb{E}\left[(I^2_{s,t})^M\right]
    & \lesssim \lvert l \rvert^{2M}(1+\lVert \omega_0\rVert^{M(2\alpha+1)})\lvert t-s\rvert^{M\left(1\wedge (\frac32- \alpha) \right)},\\
    \mathbb{E}\left[(I^4_{s,t})^M\right]
    & \lesssim \lvert l\rvert^M\lvert t-s\rvert^{M/2} \big(1+\lVert \omega_0\rVert^{M(\alpha+1)}\big).
\end{align*}
Combining them the thesis follows immediately.
\end{proof}

Thanks to the discussion before Lemma \ref{Lemma tightness galerkin} in order to obtain the required tightness in $L^2([0,T];H^{1-})\cap C([0,T];H^-)$ we need the following result.

\begin{lemma}\label{Lemma tightness scaling}
    If $\beta>3+\frac{2}{r_1}, \ s_1<\frac{1}{2}$, $ s_1r_1>1$ and $p>1$,  there exists a constant $C$ independent of $N$ such that
    \begin{align*}
        \mathbb{E}\left[\int_0^T\lVert \nabla\omega^N_s\rVert^2\, ds+ \int_0^T\lVert \omega^N_s\rVert^p\, ds + \int_0^T dt \int_0^T ds \frac{\lVert \omega^N_t-\omega^N_s\rVert_{H^{-\beta}}^{r_1}}{|t-s|^{1+r_1 s_1}}\right]\leq C.
    \end{align*}
\end{lemma}

We omit its proof since it is just a computation based on the definition of the Sobolev norms and the estimate guaranteed by Lemma \ref{solution scaled vs eigenfunctions}.
Combining the lemma above with Theorems \ref{compactness 1}, \ref{compactness 2} we have the following tightness result.

\begin{corollary}\label{corollary compactenss smag}
The family of laws of $\omega^N$ is tight on $C([0,T];H^-)\cap L^2(0,T;H^{1-}).$
\end{corollary}

\subsection{Passage to the limit}

The preliminary part in order to showing the convergence is analogous to Subsection \ref{passage to the limit Galerkin}. Arguing as in \cite{flandoli2021scaling}, by Skorohod's representation theorem, we can find, up to passing to subsequences, an auxiliary probability space, that for simplicity we continue to call $(\Omega,\mathcal{F},\mathbb{P})$, and processes $(\Tilde{\omega}^N, \tilde W^N:=\{\tilde W^{N,k}\}_{k\in \mathbb{Z}^2_0}),\ (\overline{\omega}, \tilde W:=\{\tilde W^{k}\}_{k\in \mathbb{Z}^2_0}),  $ such that \begin{align*}
    &\Tilde{\omega}^N\rightarrow \overline{\omega} \quad \text{in } C([0,T];H^-)\cap L^2(0,T;H^{1-}) \quad \mathbb{P} \mbox{-a.s.}\\
   & \tilde W^N\rightarrow \tilde W \quad \text{in } C([0,T];\mathbb{Z}^2_0)\quad \mathbb{P} \mbox{-a.s.}
\end{align*}
The convergence above from $\tilde W^N$ to $\tilde W$ can be seen as the uniform convergence of cylindrical Wiener processes $\Tilde{W}^N=\sum_{k\in \mathbb{Z}^2_0}e_k \tilde W^{N,k},\ \Tilde{W}=\sum_{k\in \mathbb{Z}^2_0}e_k \tilde W^{k}$ on a suitable Hilbert space $U_0$. Before going on, in order to identify $\overline{\omega}$ as a random variable supported on the weak solutions of equation \eqref{Ito equation} we need further integrability properties of $\overline{\omega}$. The proof of the proposition below is analogous to Proposition \ref{further estimates solution}, therefore we will omit the details.
\begin{proposition}\label{further estimates solution smagorinky}
The process $\overline{\omega}$ has weakly continuous trajectories on $L^2(\mathbb{T}^2)$ and satisfies
\begin{align*}
    \sup_{t\in [0,T]}\lVert \overline{\omega}_t\rVert^2\leq \lVert \omega_0\rVert^2 \quad\mathbb{P}\mbox{-a.s.}\\
    2\nu \int_0^T \lVert \nabla \overline{\omega}_s\rVert^2 ds \leq \lVert \omega_0\rVert^2 \quad\mathbb{P}\mbox{-a.s.}
\end{align*}
\end{proposition}

Before exploiting the convergence properties of $\omega^N$, we are interested in showing the uniqueness of weak solutions of \eqref{smagorinsky equation}. The approach we follow is the so called $H^{-1}$-method for active scalars, see for example Theorem 2 and Theorem 5 in \cite{azzam2015bounded} for other applications of this method.

\begin{lemma}\label{uniqueness smag}
    There exists at most one solution of \eqref{smagorinsky equation} in the sense of definition \ref{def well smag}.
\end{lemma}
\begin{proof}
   First note that, arguing for example as in \cite{flandoli2022heat, FlaLuoWaseda}, we can extend the weak formulation of \eqref{smagorinsky equation} in Definition \ref{def well smag} to time-dependent test functions in $ L^2(0,T;H^2)\cap W^{1,2}(0,T;L^2)$. Therefore,
   our weak formulation becomes: for any $\phi\in L^2(0,T;H^2)\cap W^{1,2}(0,T;L^2)$ and $ t\in [0,T]$, it holds
\begin{align*}
    \langle \overline{\omega}_t,\phi_t\rangle-\langle \omega_0,\phi_0\rangle &= \int_0^t\langle \overline{\omega}_s,\partial_s\phi_s\rangle\, ds + \nu \int_0^t \langle \overline{\omega}_s,\Delta \phi_s\rangle\, ds \\ &\quad + \int_0^t \langle  g(\overline{\omega}_s), \Delta\phi_s\rangle\, ds +\int_0^t \langle \overline{\omega}_s, K[\overline{\omega}_s]\cdot\nabla\phi_s \rangle\, ds.
\end{align*}

   Consider now two weak solutions $\omega, \tilde \omega\in C_w([0,T];L^2)\cap L^2(0,T;H^1)$.
    Looking at the equation and exploiting the regularity of the weak solutions, it follows that actually $\omega, \Tilde{\omega}\in C_{w}([0,T];L^2)\cap W^{1,2}(0,T;H^{-2})$. Let $w_t=\omega_t-\tilde \omega_t$, then
    \begin{align*}
        \psi=(-\Delta)^{-1}w\in C_w([0,T];H^{2})\cap L^2(0,T;H^3)\cap W^{1,2}(0,T;L^2)
    \end{align*}
   is a proper test function and we obtain
   \begin{align*}
       \frac12 \lVert w_t\rVert_{H^{-1}}^2 & =\frac12 \lVert \psi_t\rVert_{H^{1}}^2\\
       &=-\nu\int_0^t \lVert w_s\rVert^2 ds-  \int_0^t \int_{\mathbb{T}^2} \left(g(\omega_s)- g(\tilde\omega_s)\right) \left(\omega_s-\Tilde{\omega}_s\right)dxds \\
       &\quad +\int_0^t\int_{\mathbb{T}^2} w_sK[\omega_s]\cdot \nabla\psi_s \, dxds
       +\int_0^t\int_{\mathbb{T}^2}\tilde \omega_s K[w_s]\cdot \nabla\psi_s\, dx ds.
   \end{align*}
   Thanks to the fact that the function $g$ is monotone increasing, we have \begin{align*}
      \int_0^t \int_{\mathbb{T}^2} \left(g(\omega_s)-g(\tilde\omega_s)\right)\left(\omega_s-\Tilde{\omega}_s\right)dxds\geq 0.
   \end{align*}
   Next, by the definition of $\psi$ and integrating by parts,
     $$ \aligned
     \int_{\mathbb{T}^2} w_sK[\omega_s]\cdot \nabla\psi_s \, dx &= -\int_{\mathbb T^2} (\Delta \psi_s)\, K[\omega_s]\cdot \nabla\psi_s \, dx \\
     &= \int_{\mathbb T^2} (\nabla \psi_s)\cdot \nabla( K[\omega_s]\cdot \nabla\psi_s) \, dx \\
     &= \int_{\mathbb T^2} (\nabla \psi_s)\cdot \big((\nabla K[\omega_s])\cdot \nabla\psi_s\big) \, dx \\
     &\quad + \int_{\mathbb T^2} (\nabla \psi_s)\cdot \big(K[\omega_s]\cdot \nabla (\nabla\psi_s)\big) \, dx;
     \endaligned $$
   the last integral vanishes since $K[\omega_s]$ is divergence free. Therefore,
   we can proceed as in the proof of \cite[Theorem 5]{azzam2015bounded} and obtain
   \begin{align*}
       &\frac12 \lVert \psi_t\rVert_{H^{1}}^2+\nu\int_0^t \lVert w_s\rVert^2\, ds \\
       &\leq \int_0^t\int_{\mathbb{T}^2} \lvert \nabla K[\omega_s] \rvert\, \lvert\nabla\psi_s\rvert^2 \, dxds +\int_0^t \int_{\mathbb{T}^2} |\tilde \omega_s|\, |K[w_s]|\, |\nabla\psi_s|\, dx ds\\
       & \leq \int_0^t \big(\lVert \nabla K[\omega_s]\rVert\, \lVert \nabla\psi_s\rVert_{L^4}^2+ \lVert \Tilde{\omega}_s\rVert\, \lVert \nabla\psi_s\rVert_{L^4}\lVert K[\operatorname{div}\nabla \psi_s]\rVert_{L^4}\big) \,ds.
   \end{align*}
   We remark that $\nabla K:L^2 \to L^2$ and $K \operatorname{div}: L^4\to L^4$ are bounded operators, hence
   \begin{align*}
       \frac12 \lVert \psi_t\rVert_{H^{1}}^2+\nu\int_0^t \lVert w_s\rVert^2\, ds
       & \lesssim  \lVert \nabla K\rVert_{L^2\rightarrow L^2}\int_0^t \lVert \omega_s\rVert\, \lVert \nabla\psi_s\rVert_{L^4}^2\, ds \\
       &\quad + \lVert K\operatorname{div}\rVert_{L^4\rightarrow L^4}\int_0^t\lVert \Tilde{\omega}_s\rVert\, \lVert \nabla\psi_s\rVert_{L^4}^2\, ds\\
       &\lesssim \int_0^t \left( \lVert \omega_s\rVert+\lVert \Tilde{\omega}_s\rVert \right) \lVert \nabla \psi_s\rVert_{L^2}\lVert \nabla \psi_s\rVert_{H^1}\, ds
   \end{align*}
   by Sobolev embedding and interpolation. Therefore,
   \begin{align*}
       \frac12 \lVert \psi_t\rVert_{H^{1}}^2+\nu\int_0^t \lVert w_s\rVert^2\, ds
       & \lesssim \int_0^t \left( \lVert \omega_s\rVert+\lVert \Tilde{\omega}_s\rVert \right)\lVert \psi_s\rVert_{H^1}\lVert \psi_s\rVert_{H^2}\, ds\\
       & \leq C_\nu\! \int_0^t\!\! \left( \lVert \omega_s\rVert^2+\lVert \Tilde{\omega}_s\rVert^2 \right)\lVert \psi_s\rVert_{H^1}^2ds +\frac{\nu}{2}\!\int_0^t \lVert w_s\rVert^2\, ds.
   \end{align*}
   Since both $\omega$ and $\Tilde{\omega}$ belong to $C_w(0,T;L^2(\mathbb{T}^2))$, by Gr\"onwall's inequality the thesis follows.
\end{proof}

Now we are ready to provide the proof of our main theorem.

\begin{proof}[Proof of Theorem \ref{Scaling Limit Smagorinski}]
Let $\phi\in C^{\infty}(\mathbb{T}^2)$, by classical arguments for each $N\in \mathbb{N}$, $\Tilde{\omega}^N$ satisfies the following weak formulation: $\mathbb{P}$-a.s. for all $t\in [0,T]$,
\begin{align*}
    \langle \Tilde{\omega}^N_t-\omega_0,\phi\rangle
    &= \nu \int_0^t \langle \Tilde{\omega}^N_s,\Delta \phi\rangle\, ds + \int_0^t \langle g( \Tilde{\omega}^N_s), \Delta\phi\rangle\, ds\\
    & +\int_0^t \langle K[\Tilde{\omega}^N_s]\cdot\nabla\phi,\Tilde{\omega}^N_s\rangle\, ds + \sum_{k} \int_0^t\theta^{N}_k \langle \sigma_k\cdot\nabla \phi,f(\Tilde{\omega}^N_s) \rangle\, d\tilde W^{N,k}_s .
\end{align*}
    Up to passing to a further subsequence, we will show the $\mathbb{P}$-a.s. convergence, uniformly in time, of all the terms appearing above, except the martingale part; this is the only term that will present some differences with respect to the proof of Theorem \ref{thm well posed}. Therefore, we omit the treatments of the other terms which are similar to the proof of Theorem \ref{thm well posed}, and concentrate on the martingale part which will be shown to vanish in the limit, uniformly in time.

    In order to deal with the stochastic integral, applying Burkholder-Davis-Gundy inequality and using the fact that $\{\sigma_k\}_k$ is an orthonormal family of vector fields, we obtain
    \begin{align}
        J&:= \mathbb{E}\left[\sup_{t\in [0,T]}\left\lvert \sum_{k}\int_0^t\theta^{N}_k \langle \sigma_k\cdot\nabla \phi,f(\Tilde{\omega}^N_s) \rangle\, dW^{N,k}_s \right\rvert^2 \right] \notag\\
        & \lesssim \mathbb{E}\left[\sum_{k} \int_0^T (\theta^{N}_k)^2 \langle \sigma_k\cdot\nabla \phi,f(\Tilde{\omega}^N_s) \rangle^2\, ds\right]\notag\\
        & \leq \lVert \theta^N\rVert_{\ell^{\infty}}^2\, \mathbb{E} \bigg[\int_0^T\lVert f(\Tilde{\omega}^N_s) \nabla\phi \rVert^2 \, ds\bigg]\notag .
    \end{align}
Then by relation \eqref{bound f'} and Sobolev embedding theorem,
    \begin{align}
        J & \lesssim \lVert \theta^N\rVert_{\ell^{\infty}}^2\lVert \phi\rVert_{W^{1,\infty}}^2 \mathbb{E} \bigg[ \int_0^T \big\lVert 1+\lvert  \Tilde{\omega}^N_s\rvert^{\alpha+1} \big\rVert^2\, ds\bigg]\notag\\
        & \lesssim \lVert \theta^N\rVert_{\ell^{\infty}}^2\lVert \phi\rVert_{W^{1,\infty}}^2 \bigg(1+ \mathbb{E} \bigg[ \int_0^T\lVert \Tilde{\omega}^N_s\rVert_{L^{2(\alpha+1)}}^{2(\alpha+1)}\, ds \bigg]\bigg) \notag\\
        & \lesssim \lVert \theta^N\rVert_{\ell^{\infty}}^2\lVert \phi\rVert_{W^{1,\infty}}^2 \bigg(1+ \mathbb{E} \bigg[ \int_0^T\lVert \Tilde{\omega}^N_s\rVert_{H^{\frac{\alpha}{\alpha+1}}}^{2(\alpha+1)}\, ds\bigg] \bigg) \notag ,
    \end{align}
and, using interpolation inequalities and \eqref{eq:theta-N-1} yields
    \begin{align}
        J & \leq \lVert \theta^N\rVert_{\ell^{\infty}}^2\lVert \phi\rVert_{W^{1,\infty}}^2 \bigg( 1+ \mathbb{E} \bigg[ \int_0^T\lVert   \Tilde{\omega}^N_s\rVert_{H^1}^{2\alpha}\lVert \tilde\omega^N_s \rVert^2\, ds\bigg] \bigg)\notag\\
        & \lesssim \lVert \theta^N\rVert_{\ell^{\infty}}^2\lVert \phi\rVert_{W^{1,\infty}}^2 \big(1+\lVert \omega_0\rVert^{2(\alpha+1)} \big)\rightarrow 0. \notag
    \end{align}

    Summarizing the above arguments we arrive at
    \begin{align}
    \langle \overline{\omega}_t,\phi\rangle-\langle \overline{\omega}_0,\phi\rangle
    &= \nu \int_0^t \langle \overline{\omega}_s,\Delta \phi\rangle\, ds+\int_0^t \langle g( \overline{\omega}_s), \Delta\phi\rangle\, ds \notag\\
    & +\int_0^t \langle K[\overline{\omega}_s]\cdot\nabla\phi,\overline{\omega}_s\rangle\, ds\quad\mathbb{P} \mbox{-a.s. } \forall\, t\in [0,T] \label{limit relation}.
    \end{align}
    By standard density argument we can find a zero measure set $N$ such that on its complementary relation \eqref{limit relation} holds for each $\phi\in C^{\infty}$.
    By Corollary \ref{corollary compactenss smag} and Lemma \ref{uniqueness smag}, every subsequence $\mathcal{L}(\omega^{N_k})$ admits a sub-subsequence which converges to the unique limit point $\delta_{\overline{\omega}}$, where $\overline{\omega}$ is the unique deterministic solution of \eqref{smagorinsky equation}. Then, for example by \cite[Theorem 2.6]{billingsley2013convergence}, the whole sequence $\mathcal{L}(\omega^{N})$ converges weakly to $\delta_{\overline{\omega}}$.
\end{proof}

As a Corollary of Lemma \ref{uniqueness smag} and Theorem \ref{Scaling Limit Smagorinski} we have the following result.

\begin{corollary}\label{well-posed smag}
    There exists a unique solution of \eqref{smagorinsky equation} in the sense of Definition \ref{def well smag}.
\end{corollary}

\section*{Acknowledgements}

The research of the first author is funded by the European Union (ERC, NoisyFluid, No. 101053472). The second author is grateful to the National Key R\&D Program of China (No. 2020YFA0712700), the National Natural Science Foundation of China (Nos. 11931004, 12090014) and the Youth Innovation Promotion Association, CAS (Y2021002).\\
Views and opinions expressed are however those of the authors only and do not necessarily reflect those of the European Union or the European Research Council. Neither the European Union nor the granting authority can be held responsible for them.

\bibliography{demo}{}
\bibliographystyle{plain}

\end{document}